\documentclass{article}
\usepackage{amsmath, amsthm}
\usepackage{amssymb}
\newtheoremstyle{theorem}
  {10pt}          
  {10pt}  
  {\sl}  
  {\parindent}     
  {\bf}  
  {. }    
  { }    
  {}     
\newtheorem{thm}{Theorem}[section]

\newtheorem{lem}[thm]{Lemma}
\newtheorem{prop}[thm]{Proposition}
\newtheorem{defn}{Definition}[section]
\newtheorem{exam}[defn]{Example}
\newtheorem{rem}{Remark}[section]
\begin{document}
\title{\large\bf On delta and nabla Caputo fractional differences and dual identities}
\author{\small \bf Thabet Abdeljawad $^a$'$^b$  \\ {\footnotesize $^a$ Department of
Mathematics, \c{C}ankaya University, 06530, Ankara, Turkey}\\ {\footnotesize $^b$ Department of Mathematics and Physical Sciences}\\
{\footnotesize Prince Sultan University, P. O. Box 66833, Riyadh 11586, Saudi Arabia}}
\date{}
\maketitle

{\footnotesize {\noindent\bf Abstract.} We Investigate two types of dual identities for Caputo fractional   differences. The first type relates nabla and delta type fractional sums and differences. The second type represented by the Q-operator relates left and right fractional sums and differences. Two types of Caputo fractional differences are introduced, one of them (dual one) is defined so that it obeys the investigated dual identities. The relation between Rieamnn and Caputo fractional differences is investigated and the delta and nabla discrete Mittag-Leffler functions are confirmed by solving Caputo type linear fractional difference equations. A nabla integration by parts formula is obtained for Caputo fractional differences as well.
 \\


\section{Introduction}

During the last two decades, due to its widespread applications in different fields of science and engineering, fractional calculus has attracted the attention of many researchers \cite{podlubny,Samko, Kilbas,agrawal1,scalas,tr}.

 Starting from the idea of  discretizing the Cauchy integral formula, Miller and Ross \cite{Miller} and Gray and Zhang \cite{Gray}  obtained discrete versions of left type fractional integrals and derivatives, called fractional sums and differences. After then, several authors  started to deal with discrete fractional calculus \cite{Th Caputo,Ferd,Feri,Nabla,Atmodel,Nuno,TDbyparts,Gdelta,Gnabla,Gfound,THFer,Thsh}, benefiting from time scales calculus originated  in 1988 (see \cite{Adv}).

  In \cite{Th Caputo}, the concept of Caputo fractional difference was introduced and investigated. In this article we proceed deeply to investigate  Caputo fractional differences under two kinds of dual identities. The first kind relates nabla and delta type Caputo fractional differences and the second one, represented by the Q-operator, relates left and right ones. Arbitrary order Riemann and Caputo fractional differences are related as well. By the help of the previously obtained results in \cite{THFer} and \cite{Thsh} an integration by parts formula for Caputo fractional differences is originated.

 The article is organized as follows: The remaining part of this section contains  summary to some of the basic notations and definitions in delta and nabla calculus. Section 2 contains the definitions in the frame of delta and nabla fractional sums and differences in the Riemann sense. The third section contains some dual identities relating nabla and delta type fractional sums and differences in Riemann sense as previously investigated in \cite{Thsh}.  In Section 4 Caputo fractional differences are given and related to the Riemann ones. In section 5,  slightly different modified (dual) Caputo fracctional differences are introduced and investigated under  some dual identities.
Section 6 is devoted to the integration by parts for delta and nabla Caputo fractional differences. Finally, Section 7 contains Caputo type fractional dynamical equations where  a nonhomogeneous nabla Caputo fractional difference equation is solved to obtain nabla discrete versions for Mittag-Leffler functions. For the case $\alpha =1$ we obtain the discrete nabla exponential function \cite{Adv}. In addition to this, the Q-operator is used to relate left and right Caputo fractional differences in the nabla and delta case. The Q-dual identities obtained in this section expose the validity of the definition of delta and nabla right Caputo fractional differences.

For a natural number $n$, the fractional polynomial is defined by,

 \begin{equation} \label{fp}
 t^{(n)}=\prod_{j=0}^{n-1} (t-j)=\frac{\Gamma(t+1)}{\Gamma(t+1-n)},
 \end{equation}
where $\Gamma$ denotes the special gamma function and the product is
zero when $t+1-j=0$ for some $j$. More generally, for arbitrary
$\alpha$, define
\begin{equation} \label{fpg}
t^{(\alpha)}=\frac{\Gamma(t+1)}{\Gamma(t+1-\alpha)},
\end{equation}
where the convention that division at pole yields zero.
Given that the forward and backward difference operators are defined
by
\begin{equation} \label{fb}
\Delta f(t)=f(t+1)-f(t)\texttt{,}~\nabla f(t)=f(t)-f(t-1)
\end{equation}
respectively, we define iteratively the operators
$\Delta^m=\Delta(\Delta^{m-1})$ and $\nabla^m=\nabla(\nabla^{m-1})$,
where $m$ is a natural number.

 Here are some properties of the  factorial function.

\begin{lem} \label{pfp} (\cite{Ferd})
Assume the following factorial functions are well defined.

(i) $ \Delta t^{(\alpha)}=\alpha  t^{(\alpha-1)}$.

(ii) $(t-\mu)t^{(\mu)}= t^{(\mu+1)}$, where $\mu \in \mathbb{R}$.

(iii) $\mu^{(\mu)}=\Gamma (\mu+1)$.

(iv) If $t\leq r$, then $t^{(\alpha)}\leq r^{(\alpha)}$ for any
$\alpha>r$.

(v) If $0<\alpha<1$, then $ t^{(\alpha\nu)}\geq
(t^{(\nu)})^\alpha$.

(vi) $t^{(\alpha+\beta)}= (t-\beta)^{(\alpha)} t^{(\beta)}$.
\end{lem}

Also, for our purposes we list down the following two properties, the proofs of which are straightforward.

\begin{equation} \label{ou1}
\nabla_s (s-t)^{(\alpha-1)}=(\alpha-1)(\rho(s)-t)^{(\alpha-2)}.
\end{equation}

\begin{equation} \label{ou2}
\nabla_t
(\rho(s)-t)^{(\alpha-1)}=-(\alpha-1)(\rho(s)-t)^{(\alpha-2)}.
\end{equation}

For the sake of the nabla fractional calculus we have the following definition

\begin{defn} \label{rising}(\cite{Adv,Boros,Grah,Spanier})

(i) For a natural number $m$, the $m$ rising (ascending) factorial of $t$ is defined by

\begin{equation}\label{rising 1}
    t^{\overline{m}}= \prod_{k=0}^{m-1}(t+k),~~~t^{\overline{0}}=1.
\end{equation}

(ii) For any real number the $\alpha$ rising function is defined by
\begin{equation}\label{alpharising}
 t^{\overline{\alpha}}=\frac{\Gamma(t+\alpha)}{\Gamma(t)},~~~t \in \mathbb{R}-~\{...,-2,-1,0\},~~0^{\mathbb{\alpha}}=0
\end{equation}

\end{defn}

Regarding the rising factorial function we observe the following:

(i) \begin{equation}\label{oper}
    \nabla (t^{\overline{\alpha}})=\alpha t^{\overline{\alpha-1}}
\end{equation}

 (ii)
 \begin{equation}\label{oper2}
    (t^{\overline{\alpha}})=(t+\alpha-1)^{(\alpha)}.
 \end{equation}

(iii)
\begin{equation}\label{oper3}
   \Delta_t (s-\rho(t))^{\overline{\alpha}}= -\alpha  (s-\rho(t))^{\overline{\alpha-1}}
\end{equation}

\textbf{Notation}:
\begin{enumerate}
\item[$(i)$] For a real $\alpha>0$, we set $n=[\alpha]+1$, where $[\alpha]$ is the greatest integer less than $\alpha$.

\item[$(ii)$] For real numbers $a$ and $b$, we denote $\mathbb{N}_a=\{a,a+1,...\}$ and $~_{b}\mathbb{N}=\{b,b-1,...\}$.

\item[$(iii)$]For $n \in \mathbb{N}$ and real $a$, we denote
$$ _{\circleddash}\Delta^n f(t)\triangleq (-1)^n\Delta^n f(t).$$
\item[$(iv)$]For $n \in \mathbb{N}$ and real $b$, we denote
                   $$ \nabla_{\circleddash}^n f(t)\triangleq (-1)^n\nabla^n f(t).$$
\end{enumerate}

\section{Definitions and essential lemmas}

\begin{defn} \label{fractional sums}
Let $\sigma(t)=t+1$ and $\rho(t)=t-1$ be the forward and backward jumping operators, respectively. Then

(i) The (delta) left fractional sum of order $\alpha>0$ (starting from $a$) is defined by:
\begin{equation}\label{dls}
    \Delta_a^{-\alpha} f(t)=\frac{1}{\Gamma(\alpha)} \sum_{s=a}^{t-\alpha}(t-\sigma(s))^{(\alpha-1)}f(s),~~t \in \mathbb{N}_{a+\alpha}.
\end{equation}

(ii) The (delta) right fractional sum of order $\alpha>0$ (ending at  $b$) is defined by:
\begin{equation}\label{drs}
   ~_{b}\Delta^{-\alpha} f(t)=\frac{1}{\Gamma(\alpha)} \sum_{s=t+\alpha}^{b}(s-\sigma(t))^{(\alpha-1)}f(s)=\frac{1}{\Gamma(\alpha)} \sum_{s=t+\alpha}^{b}(\rho(s)-t)^{(\alpha-1)}f(s),~~t \in ~_{b-\alpha}\mathbb{N}.
\end{equation}

(iii) The (nabla) left fractional sum of order $\alpha>0$ (starting from $a$) is defined by:
\begin{equation}\label{nlf}
  \nabla_a^{-\alpha} f(t)=\frac{1}{\Gamma(\alpha)} \sum_{s=a+1}^t(t-\rho(s))^{\overline{\alpha-1}}f(s),~~t \in \mathbb{N}_{a+1}.
\end{equation}

(iv)The (nabla) right fractional sum of order $\alpha>0$ (ending at $b$) is defined by:
\begin{equation}\label{nrs}
   ~_{b}\nabla^{-\alpha} f(t)=\frac{1}{\Gamma(\alpha)} \sum_{s=t}^{b-1}(s-\rho(t))^{\overline{\alpha-1}}f(s)=\frac{1}{\Gamma(\alpha)} \sum_{s=t}^{b-1}(\sigma(s)-t)^{\overline{\alpha-1}}f(s),~~t \in ~_{b-1}\mathbb{N}.
\end{equation}
\end{defn}

Regarding the delta left fractional sum we observe the following:

(i) $\Delta_a^{-\alpha}$ maps functions defined on $\mathbb{N}_a$ to
functions defined on $\mathbb{N}_{a+\alpha}$.

(ii) $u(t)=\Delta_a^{-n}f(t)$, $n \in \mathbb{N}$, satisfies the
initial value problem
\begin{equation} \label{ivpf}
\Delta^n u(t)=f(t),~~t\in N_a,~u(a+j-1)=0,~ j=1,2,...,n.
\end{equation}

(iii) The Cauchy function $\frac{(t-\sigma(s))^{(n-1)}}{(n-1)!}$
vanishes at $s=t-(n-1),...,t-1$.

\indent

Regarding the delta right fractional sum we observe the following:

(i)  $~_{b}\Delta^{-\alpha}$ maps functions defined on $_{b}\mathbb{N}$ to
functions defined on $_{b-\alpha}\mathbb{N}$.

(ii) $u(t)=~_{b}\Delta^{-n}f(t)$, $n \in \mathbb{N}$, satisfies the
initial value problem
\begin{equation} \label{ivpb}
\nabla_\ominus^n u(t)=f(t),~~t\in~ _{b}N,~u(b-j+1)=0,~ j=1,2,...,n.
\end{equation}

(iii) the Cauchy function $\frac{(\rho(s)-t)^{(n-1)}}{(n-1)!}$
vanishes at $s=t+1,t+2,...,t+(n-1)$.

\indent

Regarding the nabla left fractional sum we observe the following:

(i) $ \nabla_a^{-\alpha}$ maps functions defined on $\mathbb{N}_a$ to functions defined on $\mathbb{N}_{a}$.

(ii)$ \nabla_a^{-n}f(t)$ satisfies the n-th order discrete initial value problem

\begin{equation}\label{s1}
    \nabla^n y(t)=f(t),~~~\nabla^i y(a)=0,~~i=0,1,...,n-1
\end{equation}

(iii) The Cauchy function $\frac{(t-\rho(s))^{\overline{n-1}}}{\Gamma(n)}$ satisfies $\nabla^n y(t)=0$.

\indent

Regarding the nabla right fractional sum we observe the following:

(i) $ ~_{b}\nabla^{-\alpha}$ maps functions defined on $~_{b}\mathbb{N}$ to functions defined on $~_{b}\mathbb{N}$.

(ii)$ ~_{b}\nabla^{-n}f(t)$ satisfies the n-th order discrete initial value problem

\begin{equation}\label{s2}
    ~_{\ominus}\Delta^n y(t)=f(t),~~~ ~_{\ominus}\Delta^i y(b)=0,~~i=0,1,...,n-1.
\end{equation}
The proof can be done inductively. Namely, assuming it is true for $n$, we have
\begin{equation}\label{t1}
    ~_{\ominus}\Delta^{n+1} ~_{b}\nabla^{-(n+1)}f(t)=~_{\ominus}\Delta^{n}[-\Delta ~_{b}\nabla^{-(n+1)}f(t)].
\end{equation}

By the help of (\ref{oper3}), it follows that
\begin{equation}\label{t2}
  ~_{\ominus}\Delta^{n+1} ~_{b}\nabla^{-(n+1)}f(t)=  ~_{\ominus}\Delta^{n} ~_{b}\nabla^{-n}f(t)=f(t).
\end{equation}
The other part is clear by using the convention that $\sum_{k=s+1}^s=0$.

(iii) The Cauchy function $\frac{(s-\rho(t))^{\overline{n-1}}}{\Gamma(n)}$ satisfies $_{\ominus}\Delta^n y(t)=0$.

\indent

\begin{defn} \label{fractional differences}
(i)\cite{Miller}  The (delta) left fractional difference of order $\alpha>0$ (starting from $a$ ) is defined by:
\begin{equation}\label{dls}
    \Delta_a^{\alpha} f(t)=\Delta^n \Delta_a^{-(n-\alpha)} f(t)= \frac{\Delta^n}{\Gamma(n-\alpha)} \sum_{s=a}^{t-(n-\alpha)}(t-\sigma(s))^{(n-\alpha-1)}f(s),~~t \in \mathbb{N}_{a+(n-\alpha)}
\end{equation}

(ii) \cite{TDbyparts} The (delta) right fractional difference of order $\alpha>0$ (ending at  $b$ ) is defined by:
\begin{equation}\label{drd}
   ~_{b}\Delta^{\alpha} f(t)=  \nabla_{\circleddash}^n  ~_{b}\Delta^{-(n-\alpha)}f(t)=\frac{(-1)^n \nabla ^n}{\Gamma(n-\alpha)} \sum_{s=t+(n-\alpha)}^{b}(s-\sigma(t))^{(n-\alpha-1)}f(s),~~t \in ~_{b-(n-\alpha)}\mathbb{N}
\end{equation}

(iii) The (nabla) left fractional difference of order $\alpha>0$ (starting from $a$ ) is defined by:
\begin{equation}\label{nld}
  \nabla_a^{\alpha} f(t)=\nabla^n \nabla_a^{-(n-\alpha)}f(t)= \frac{\nabla^n}{\Gamma(n-\alpha)} \sum_{s=a+1}^t(t-\rho(s))^{\overline{n-\alpha-1}}f(s),~~t \in \mathbb{N}_{a+1}
\end{equation}

(iv) The (nabla) right fractional difference of order $\alpha>0$ (ending at $b$ ) is defined by:
\begin{equation}\label{nrd}
   ~_{b}\nabla^{\alpha} f(t)= ~_{\circleddash}\Delta^n ~_{b}\nabla^{-(n-\alpha)}f(t) =\frac{(-1)^n\Delta^n}{\Gamma(n-\alpha)} \sum_{s=t}^{b-1}(s-\rho(t))^{\overline{n-\alpha-1}}f(s),~~t \in ~ _{b-1}\mathbb{N}
\end{equation}

\end{defn}

Regarding the domains of the fractional type differences we observe:

(i) The delta left fractional difference $\Delta_a^\alpha$ maps functions defined on $\mathbb{N}_a$ to functions defined on $\mathbb{N}_{a+(n-\alpha)}$.

(ii) The delta right fractional difference $~_{b}\Delta^\alpha$ maps functions defined on $~_{b}\mathbb{N}$ to functions defined on $~_{b-(n-\alpha)}\mathbb{N}$.

(iii) The nabla left fractional difference $\nabla_a^\alpha$ maps functions defined on $\mathbb{N}_a$ to functions defined on $\mathbb{N}_{a+n}$ (on $\mathbb{N}_a$ if we think $f$ defined at some points before $a$).

(iv)  The nabla right fractional difference $~_{b}\nabla^\alpha$ maps functions defined on $~_{b}\mathbb{N}$ to functions defined on $~_{b-n}\mathbb{N}$ (on $_{b}\mathbb{N}$ if we think $f$ defined at some points after $b$).

 \begin{lem} \label{ATO} \cite{Ferd}
 For any $\alpha >0$, the following equality holds:
 $$ \Delta_a^{-\alpha} \Delta f(t)=  \Delta \Delta_a^{-\alpha}f(t)-\frac{(t-a)^{\overline{\alpha-1}}}  {\Gamma(\alpha)} f(a).$$
 \end{lem}

 \begin{lem} \label{TD} \cite{TDbyparts}
 For any $\alpha >0$, the following equality holds:
 $$~_{b} \Delta^{-\alpha} \nabla_{\circleddash} f(t)=  \nabla_{\circleddash} ~_{b} \Delta^{-\alpha}f(t)-\frac{(b-t)^{\overline{\alpha-1}}}  {\Gamma(\alpha)} f(b).$$
 \end{lem}

\begin{lem} \label{At} \cite{Gronwall}
For any $\alpha >0$, the following equality holds:
\begin{equation} \label{At1}
\nabla _{a+1}^{-\alpha} \nabla f(t)= \nabla \nabla_a^{-\alpha}f(t)-\frac{(t-a+1)^{\overline{\alpha-1}}  }{\Gamma(\alpha)}f(a)
\end{equation}
\end{lem}

 The result of Lemma \ref{At}  was obtained in \cite{Gronwall} by applying the nabla left fractional sum starting from $a$ not from $a+1$. Next will provide the version of Lemma \ref{At} by applying the definition in this article. Actually, the nabla fractional sums defined in this article and those in \cite{Gronwall} are related. For more details we refer to \cite{THFer}.

\begin{lem} \label{AtT} (see \cite{THFer} and \cite{Thsh})
For any $\alpha >0$, the following equality holds:
\begin{equation} \label{AtT1}
\nabla _a^{-\alpha} \nabla f(t)= \nabla \nabla_a^{-\alpha}f(t)-\frac{(t-a)^{\overline{\alpha-1}}  }{\Gamma(\alpha)}f(a).
\end{equation}
\end{lem}

\begin{rem}\label{lforany}(see \cite{THFer} and \cite{Thsh})
Let $\alpha>0$ and $n=[\alpha]+1$. Then, by the help of Lemma
 \ref{AtT} we  have
\begin{equation}\label{lforany1}
\nabla \nabla_a^\alpha f(t)=\nabla \nabla^n(\nabla_a^{-(n-\alpha)}f(t))= \nabla^n (\nabla \nabla_a^{-(n-\alpha)}f(t)).
\end{equation}
or
\begin{equation}\label{lforany11}
\nabla \nabla_a^\alpha f(t)=\nabla^n [\nabla_a^{-(n-\alpha)}\nabla
f(t)+\frac{(t-a)^{\overline{n-\alpha-1}}} {\Gamma(n-\alpha)}f(a)]
\end{equation}
Then, using the identity
\begin{equation}\label{forany2}
\nabla^n \frac{(t-a)^{\overline{n-\alpha-1}}} {\Gamma(n-\alpha)}=
\frac{(t-a)^{\overline{-\alpha-1}}} {\Gamma(-\alpha)}
\end{equation}
we infer that (\ref{AtT1}) is valid for any real $\alpha$.
\end{rem}

By the help of Lemma \ref{AtT}, Remark \ref{lforany} and the identity $\nabla (t-a)^{\overline{\alpha-1}}=(\alpha-1)(t-a)^{\overline{\alpha-2}}$, we arrive inductively at the following generalization.

 \begin{thm} \label{LNg}(see \cite{THFer} and \cite{Thsh})
 For any real number $\alpha$ and any positive integer $p$, the
following equality holds:
\begin{equation} \label{LNg1}
\nabla_a^{-\alpha}~\nabla^p f(t)=\nabla^p
\nabla_a^{-\alpha}f(t)-\sum_{k=0}^{p-1}\frac{(t-a)^{\overline{\alpha-p+k}}}
{\Gamma(\alpha+k-p+1)}\nabla^k f(a).
\end{equation}
where $f$ is defined on $\mathbb{N}_a$ and some points before $a$ .
 \end{thm}

\begin{lem} \label{RN}(see \cite{THFer} and \cite{Thsh})
For any $\alpha >0$, the following equality holds:
\begin{equation}\label{RN1}
    _{b}\nabla^{-\alpha}~ _{\circleddash}\Delta f(t)= ~ _{\circleddash}\Delta~ _{b}\nabla^{-\alpha} f(t)-\frac{(b-t)^{\overline{\alpha-1}}}  {\Gamma(\alpha)} f(b).
\end{equation}
\end{lem}

\begin{rem}\label{forany}(see \cite{THFer} and \cite{Thsh})
Let $\alpha>0$ and $n=[\alpha]+1$. Then, by the help of Lemma
 \ref{RN} we can have
\begin{equation}\label{forany1}
~_{\ominus}\Delta ~_{b}\nabla^\alpha f(t)=~_{\ominus}\Delta
~_{\circleddash}\Delta^n(~_{b}\nabla^{-(n-\alpha)}f(t))=~_{\circleddash}\Delta^n (~_{\circleddash}\Delta ~_{b}\nabla^{-(n-\alpha)}f(t))
\end{equation}
or
\begin{equation}\label{forany11}
~_{\circleddash}\Delta ~_{b}\nabla^\alpha f(t)=~_{\circleddash}\Delta^n [~_{b}\nabla^{-(n-\alpha)}~_{\circleddash}\Delta
f(t)+\frac{(b-t)^{\overline{n-\alpha-1}}} {\Gamma(n-\alpha)}f(b)]
\end{equation}
Then, using the identity
\begin{equation}\label{forany2}
~_{\circleddash}\Delta^n \frac{(b-t)^{\overline{n-\alpha-1}}} {\Gamma(n-\alpha)}=
\frac{(b-t)^{\overline{-\alpha-1}}} {\Gamma(-\alpha)}
\end{equation}
we infer that (\ref{RN1}) is valid for any real $\alpha$.
\end{rem}

By the help of Lemma \ref{RN}, Remark \ref{forany} and the identity $\Delta (b-t)^{\overline{\alpha-1}}=-(\alpha-1)(b-t)^{\overline{\alpha-2}}$, if we follow inductively we arrive at the following generalization.

 \begin{thm} \label{RNg}(see \cite{THFer} and \cite{Thsh})
 For any real number $\alpha$ and any positive integer $p$, the
following equality holds:
\begin{equation} \label{RNg1}
~_{b}\nabla^{-\alpha}~_{\circleddash}\Delta^p f(t)=~_{\circleddash}\Delta^p
~_{b}\nabla^{-\alpha}f(t)-\sum_{k=0}^{p-1}\frac{(b-t)^{\overline{\alpha-p+k}}}
{\Gamma(\alpha+k-p+1)}~_{\ominus}\Delta^k f(b)
\end{equation}
where $f$ is defined on $_{b}N$ and some points after $b$.
 \end{thm}

\section{Dual identities for  fractional sums and Riemann fractional  differences}

The dual relations for left fractional sums and differences were investigated in \cite{Nabla}. Indeed, the following two lemmas are dual relations between the delta left fractional sums (differences) and the nabla left fractional sums (differences).

\begin{lem} \label{left dual}\cite{Nabla}
Let $0\leq n-1< \alpha \leq n$ and let $y(t)$ be defined on $\mathbb{N}_a$. Then the following statements are valid.

(i)$ (\Delta_a^\alpha) y(t-\alpha)= \nabla_a^\alpha y(t)$ for $t \in \mathbb{N}_{n+a}$.

(ii) $ (\Delta_a^{-\alpha}) y(t+\alpha)= \nabla_a^{-\alpha} y(t)$ for $t \in \mathbb{N}_a$.
\end{lem}

\begin{lem} \label{left dual2}\cite{Nabla}
Let $0\leq n-1< \alpha \leq n$ and let $y(t)$ be defined on $\mathbb{N}_{\alpha-n}$. Then the following statements are valid.

(i)$ \Delta_{\alpha-n}^\alpha y(t)= (\nabla_{\alpha-n}^\alpha y)(t+\alpha)$ for $t \in \mathbb{N}_{-n}$.

(ii) $ \Delta_{\alpha-n}^{-(n-\alpha)} y(t)= (\nabla_{\alpha-n}^{-(n-\alpha)} y)(t-n+\alpha)$ for $t \in \mathbb{N}_0$.
\end{lem}
 We remind that the above two dual lemmas for left fractional sums and differences were obtained when the nabla left fractional sum was defined by

 \begin{equation}\label{remind}
  \nabla_a^{-\alpha} f(t)=\frac{1}{\Gamma(\alpha)} \sum_{s=a}^t(t-\rho(s))^{\overline{\alpha-1}}f(s),~~t \in \mathbb{N}_{a}
\end{equation}

Now, in analogous to Lemma \ref{left dual} and Lemma \ref{left dual2}, for the right fractional summations and differences the author in \cite{Thsh} obtained:

\begin{lem} \label{right dual}
Let $y(t)$ be defined on $~_{b+1}\mathbb{N}$. Then the following statements are valid.

(i)$ (~_{b}\Delta^\alpha) y(t+\alpha)= ~_{b+1}\nabla^\alpha y(t)$ for $t \in ~_{b-n}\mathbb{N}$.

(ii) $ (~_{b}\Delta^{-\alpha}) y(t-\alpha)= ~_{b+1}\nabla^{-\alpha} y(t)$ for $t \in ~_{b}\mathbb{N}$.
\end{lem}

\begin{lem} \label{right dual2} \cite{Thsh}
Let $0\leq n-1< \alpha \leq n$ and let $y(t)$ be defined on $~_{n-\alpha}\mathbb{N}$. Then the following statements are valid.

(i) $$~_{n-\alpha}\Delta^\alpha y(t)=~_{n-\alpha+1}\nabla^\alpha y(t-\alpha),~~~t \in ~_{n}\mathbb{N}$$

(ii)$$~_{n-\alpha}\Delta^{-(n-\alpha)} y(t)=~_{n-\alpha+1}\nabla^{-(n-\alpha)} y(t+n-\alpha),~~~t \in ~_{0}\mathbb{N}$$
\end{lem}

\begin{proof} We prove (i), the proof of (ii) is similar. By the definition of right nabla difference  we have

$$~_{n-\alpha+1}\nabla^\alpha y(t-\alpha)=~_{a}\Delta^n \frac{1}{\Gamma(n-\alpha)}\sum_{s=t-\alpha}^{n-\alpha}(s-\rho(t-\alpha))^{\overline{n-\alpha-1}}y(s)=$$

\begin{equation}\label{rig1}
  ~_{a}\Delta^n \frac{1}{\Gamma(n-\alpha)}\sum_{s=t-\alpha}^{n-\alpha}(s-\rho(t-\alpha))^{\overline{n-\alpha-1}}y(s)= \nabla_b^n \frac{1}{\Gamma(n-\alpha)}\sum_{s=t+n-\alpha}^{n-\alpha}(s-\rho(t+n-\alpha))^{\overline{n-\alpha-1}}y(s)
\end{equation}
By using (\ref{oper2}) it follows that

\begin{equation}\label{rig2}
~_{n-\alpha+1}\nabla^\alpha y(t-\alpha)=\nabla_b^n \frac{1}{\Gamma(n-\alpha)}\sum_{s=t+n-\alpha}^{n-\alpha}(s-\sigma(t))^{(n-\alpha-1)}y(s)=~_{n-\alpha}\Delta^\alpha y(t)
\end{equation}
\end{proof}
Note that the above two dual lemmas for right fractional  differences can not be obtained if we apply the definition of the delta right fractional difference introduced in \cite{Nuno} and \cite{Atmodel}.

\begin{lem}\label{power} \cite{TDbyparts}
Let $\alpha>0,~\mu>0$. Then,
\begin{equation} \label{power1}
~_{b-\mu}\Delta^{-\alpha}(b-t)^{(\mu)}=\frac{\Gamma(\mu+1)}
{\Gamma(\mu+\alpha+1)}(b-t)^{(\mu+\alpha)}
\end{equation}
\end{lem}
The following commutative property for delta right fractional sums is Theorem 9 in \cite{TDbyparts}.

\begin{thm}\label{bcomp}
Let $\alpha>0,~\mu>0$. Then, for all $t$ such that $t\equiv
b-(\mu+\alpha)~(mod~1)$, we have
\begin{equation}\label{bcomp1}
~_{b}\Delta^{-\alpha}[~_{b}\Delta^{-\mu}
f(t)]=~_{b}\Delta^{-(\mu+\alpha)}f(t)=~_{b}\Delta^{-\mu}[~_{b}\Delta^{-\alpha}
f(t)]
\end{equation}
where $f$ is defined on $_{b}N$.
\end{thm}

\begin{prop}\label{semi} \cite{Thsh}
Let $f$ be a real valued function defined on $~_{b}\mathbb{N}$, and let $\alpha, \beta >0$. Then

\begin{equation}\label{semi1}
    ~_{b}\nabla^{-\alpha}[~_{b}\nabla^{-\beta}f(t)]= ~_{b}\nabla^{-(\alpha+\beta)}f(t)=  ~_{b}\nabla^{-\beta}[~_{b}\nabla^{-\alpha}f(t)]
\end{equation}
\end{prop}

\begin{proof}
The proof follows by applying Lemma \ref{right dual}(ii) and Theorem \ref{bcomp} above. Indeed,

$$ ~_{b}\nabla^{-\alpha}[~_{b}\nabla^{-\beta}f(t)]=  ~_{b}\nabla^{-\alpha} ~_{b-1}\Delta ^{-\beta}f(t-\beta)= $$

\begin{equation}\label{semi2}
  ~_{b-1}\Delta^{-\alpha} ~_{b-1}\Delta^{-\beta} f(t-(\alpha+\beta))=   ~_{b-1}\Delta^{-(\alpha+\beta)} f(t-(\alpha+\mu)) =~_{b}\nabla^{-(\alpha+\beta)}y(t)
\end{equation}
\end{proof}

The following power rule for nabla right fractional differences plays an important rule.

\begin{prop} \label{nabla right power} (\cite{THFer}, \cite{Thsh})
Let $\alpha >0,~ \mu >-1$. Then, for $t \in ~_{b}\mathbb{N}$ , we have
\begin{equation}\label{pnr1}
~ _{b}\nabla ^{-\alpha} (b-t)^{\overline{\mu}}=\frac{\Gamma(\mu+1) } {\Gamma(\alpha+\mu+1)} (b-t)^{\overline{\alpha+\mu}}
\end{equation}

\end{prop}

\begin{proof}
By the dual formula (ii) of Lemma \ref{right dual}, we have
$$~ _{b}\nabla ^{-\alpha} (b-t)^{\overline{\mu}}= ~_{b-1}\Delta^{-\alpha} (b-r)^{\overline{\mu}}|_{r=t-\alpha}=$$
\begin{equation}\label{pnr2}
\frac{1}{\Gamma(\alpha)} \sum_{s=t}^{b-1}(s-t+\alpha-1)^{(\alpha-1)}(b-s)^{\overline{\mu}}.
\end{equation}
Then by the identity $t^{\overline{\alpha}}=(t+\alpha-1)^{(\alpha-1)}$ and using the change of variable $r=s-\mu+1$, it follows that
$$~ _{b}\nabla ^{-\alpha} (b-t)^{\overline{\mu}}=$$
\begin{equation}\label{pnr3}
  \frac{1}{\Gamma(\alpha)} \sum_{r=t-\mu+1}^{b-\mu} (r-\sigma(t-\alpha-\mu+1))^{(\alpha-1)} (b-r)^{\overline{\mu}}=(~_{b-\mu}\Delta^{-\alpha} (b-u)^{\overline{\mu}})|_{u=-\alpha-\mu+1+t}.
\end{equation}
Which  by Lemma \ref{power} leads to

$$~ _{b}\nabla ^{-\alpha} (b-t)^{\overline{\mu}}=$$
\begin{equation}\label{pnr4}
\frac{\Gamma(\mu+1) } {\Gamma(\alpha+\mu+1)}(b-t+\alpha+\mu-1)^{(\alpha+\mu)}=\frac{\Gamma(\mu+1) } {\Gamma(\alpha+\mu+1)} (b-t)^{\overline{\alpha+\mu}}
\end{equation}
\end{proof}

Similarly, for the nabla left fractional sum we can have the following power formula and exponent law

\begin{prop} \label{power nabla left}(see \cite{THFer} and \cite{Thsh})
Let $\alpha >0,~ \mu >-1$. Then, for $t \in \mathbb{N}_a$ , we have
\begin{equation}\label{pnl1}
\nabla_a ^{-\alpha} (t-a)^{\overline{\mu}}=\frac{\Gamma(\mu+1) } {\Gamma(\alpha+\mu+1)} (t-a)^{\overline{\alpha+\mu}}
\end{equation}

\end{prop}

\begin{prop}\label{lsemi}(see \cite{THFer} and \cite{Thsh})
Let $f$ be a real valued function defined on $\mathbb{N}_a$, and let $\alpha, \beta >0$. Then

\begin{equation}\label{lsemi1}
    \nabla_a^{-\alpha}[\nabla_a^{-\beta}f(t)]= \nabla_a^{-(\alpha+\beta)}f(t)=  \nabla_a^{-\beta}[\nabla_a^{-\alpha}f(t)]
\end{equation}
\end{prop}

\begin{proof}
The proof can be achieved as in Theorem 2.1 \cite{Nabla}, by expressing the left hand side of (\ref{lsemi1}), interchanging the order of summation and using the power formula (\ref{pnl1}). Alternatively, the proof can be done by following as in the proof of Proposition \ref{semi} with the help of the dual formula for left fractional sum in Lemma \ref{left dual} after its arrangement according to our definitions.
\end{proof}

\section{Caputo fractional differences}

In analogous to the usual fractional calculus we can formulate the following definition
\begin{defn}\label{cd}
Let  $\alpha>0,~\alpha \notin \mathbb{N}$. Then,

(i)\cite{Th Caputo} the delta $\alpha-$order Caputo left  fractional difference of a function $f$ defined on $\mathbb{N}_a$  is defined by
\begin{equation} \label{rd}
~^{C}\Delta_a^\alpha f(t)\triangleq\Delta_a ^{-(n-\alpha)}\Delta
^nf(t)=\frac{1}{\Gamma(n-\alpha)}
\sum_{s=a}^{t-(n-\alpha)}(t-\sigma(s))^{(n-\alpha-1)}\Delta_s^nf(s)
\end{equation}

(ii) \cite{Th Caputo} the delta $\alpha-$ order Caputo right  fractional difference of a function $f$ defined on $~_{b}\mathbb{N}$  is defined by

\begin{equation} \label{ld}
~^{C}_{b}\Delta^\alpha f(t)\triangleq ~_{b}\Delta ^{-(n-\alpha)}\nabla_{\ominus}^nf(t)=\frac{1}{\Gamma(n-\alpha)}
\sum_{s=t+(n-\alpha)}^b(s-\sigma(t))^{(n-\alpha-1)}\nabla_{\ominus}^nf(s)
\end{equation}
where $n=[\alpha]+1$.

If $\alpha =n\in \mathbb{N}$, then
$$~^{C}\Delta_a^\alpha f(t)\triangleq \Delta^n f(t)~~\texttt{and}~
~^{C}_{b}\Delta^\alpha f(t)\triangleq \nabla_b^n f(t)$$
\end{defn}

(iii) the nabla $\alpha-$order Caputo left  fractional difference of a function $f$ defined on $\mathbb{N}_a$ and some points before $a$, is defined by

\begin{equation}\label{cnl}
  ~^{C}\nabla_a^\alpha f(t)\triangleq \nabla_a ^{-(n-\alpha)}\nabla
^nf(t)=\frac{1}{\Gamma(n-\alpha)}
\sum_{s=a+1}^{t-(n-\alpha)}(t-\rho(s))^{\overline{n-\alpha-1}}\nabla^nf(s)
\end{equation}

(iv) the nabla $\alpha-$order Caputo right fractional difference of a function $f$ defined on $~_{b}\mathbb{N}$ and some points after $b$, is defined by

\begin{equation}\label{cnr}
  ~^{C}_{b}\nabla^\alpha f(t)\triangleq ~_{b}\nabla ^{-(n-\alpha)}~{a}\Delta
^nf(t)=\frac{1}{\Gamma(n-\alpha)}
\sum_{s=t}^{b-1}(s-\rho(t))^{\overline{n-\alpha-1}}~_{\ominus}\Delta^nf(s)
\end{equation}

If $\alpha =n\in \mathbb{N}$, then
$$~^{C}\nabla_a^\alpha f(t)\triangleq \nabla^n f(t)~~~\texttt{and}~~
~^{C}_{b}\nabla^\alpha f(t)\triangleq ~_{a}\Delta^n f(t)$$

It is clear that $~^{C}\Delta_a^\alpha$ maps functions defined on
$\mathbb{N}_a$ to functions defined on $\mathbb{N}_{a+(n-\alpha)}$, and that
$~^{C}_{b}\Delta^\alpha$ maps functions defined on $_{b}\mathbb{N}$ to functions
defined on $_{b-(n-\alpha)}\mathbb{N}$.
Also, it is clear that the nabla left fractional difference $\nabla_a^\alpha$ maps functions defined on $\mathbb{N}_a$ to functions defined on $\mathbb{N}_{a+1-n}$ and the nabla right fractional difference $~_{b}\nabla^\alpha$ maps functions defined on $~_{b}\mathbb{N}$ to functions defined on $~_{b-1+n}\mathbb{N}$.

Riemann and Caputo delta fractional differences are related by the following theorem

\begin{thm} \label{relate} \cite{Th Caputo}
For any $\alpha>0$, we have
\begin{equation}\label{relate1}
~^{C}\Delta_a^\alpha f(t)=\Delta_a^\alpha f(t)-\sum_{k=0}^{n-1}
\frac{(t-a)^{(k-\alpha)}}{\Gamma(k-\alpha+1)}\Delta^k f(a)
\end{equation}
and
\begin{equation}\label{relate2}
~_{b}^{C}\Delta^\alpha f(t)=~_{b}\Delta^\alpha f(t)-\sum_{k=0}^{n-1}
\frac{(b-t)^{(k-\alpha)}}{\Gamma(k-\alpha+1)}\nabla_{\ominus}^k f(b).
\end{equation}
In particular, when $0<\alpha<1$, we have
\begin{equation}\label{relate2}
~^{C}\Delta_a f(t)=\Delta_a^\alpha f(t)-
\frac{(t-a)^{(-\alpha)}}{\Gamma(1-\alpha)} f(a).
\end{equation}

\begin{equation}\label{relate4}
~_{b}^{C}\Delta f(t)=~_{b}\Delta^\alpha f(t)-
\frac{(b-t)^{(-\alpha)}}{\Gamma(1-\alpha)} f(b)
\end{equation}
\end{thm}
One can note that the Riemann and Caputo fractional differences, for
$0<\alpha<1$, coincide when $f$ vanishes at the end points.

The following identity is useful to transform delta type Caputo fractional
difference equations  into fractional summations.

\begin{prop} \label{trans}\cite{Th Caputo}
Assume $\alpha>0$ and $f$ is defined on suitable domains $\mathbb{N}_a$
and $_{b}\mathbb{N}$. Then
\begin{equation}\label{trans1}
\Delta_{a+(n-\alpha)}^{-\alpha} ~^{C}\Delta_a^\alpha
f(t)=f(t)-\sum_{k=0}^{n-1}\frac{(t-a)^{(k)}}{k!}\Delta^kf(a)
\end{equation}
and
\begin{equation}\label{trans2}
~_{b-(n-\alpha)}\Delta^{-\alpha} ~_{b}^{C}\Delta^\alpha
f(t)=f(t)-\sum_{k=0}^{n-1}\frac{(b-t)^{(k)}}{k!}\nabla_{\ominus}^kf(b)
\end{equation}
In particular, if $0<\alpha\leq1$ then
\begin{equation}\label{trans3}
\Delta_{a+(n-\alpha)}^{-\alpha} ~^{C}\Delta_a^\alpha f(t)= f(t)-f(a)~~\texttt{and}~~
 _{b-(n-\alpha)}\Delta^{-\alpha}
~^{C}_{b}\Delta^\alpha f(t)=f(t)-f(b).
\end{equation}

\end{prop}

Similar to what we have  above, for the nabla fractional differences we obtain

\begin{thm} \label{nabla relate}
For any $\alpha>0$, we have
\begin{equation}\label{nrelate1}
~^{C}\nabla_a^\alpha f(t)=\nabla_a^\alpha f(t)-\sum_{k=0}^{n-1}
\frac{(t-a)^{\overline{k-\alpha}}}{\Gamma(k-\alpha+1)}\nabla^k f(a)
\end{equation}
and
\begin{equation}\label{nrelate2}
~_{b}^{C}\nabla^\alpha f(t)=~_{b}\nabla^\alpha f(t)-\sum_{k=0}^{n-1}
\frac{(b-t)^{\overline{k-\alpha}}}{\Gamma(k-\alpha+1)}~_{\ominus}\Delta^k f(b).
\end{equation}
In particular, when $0<\alpha<1$, we have
\begin{equation}\label{nrelate2}
~^{C}\nabla_a^\alpha f(t)=\nabla_a^\alpha f(t)-
\frac{(t-a)^{\overline{-\alpha}}}{\Gamma(1-\alpha)} f(a)
\end{equation}
and
\begin{equation}\label{nrelate4}
~_{b}^{C}\nabla^\alpha f(t)=~_{b}\nabla^\alpha f(t)-
\frac{(b-t)^{\overline{-\alpha}}}{\Gamma(1-\alpha)} f(b)
\end{equation}
\end{thm}
\begin{proof}
The proof follows by replacing $\alpha$ by $n-\alpha$ and $p$ by $n$ in Theorem \ref{LNg} and Theorem \ref{RNg}, respectively.
\end{proof}

One can see that the nabla Riemann and Caputo fractional differences, for
$0<\alpha<1$, coincide when $f$ vanishes at the end points.

\begin{prop} \label{nabla trans}
Assume $\alpha>0$ and $f$ is defined on suitable domains $\mathbb{N}_a$
and $_{b}\mathbb{N}$. Then
\begin{equation}\label{ntrans1}
\nabla_a^{-\alpha} ~^{C}\nabla_a^\alpha
f(t)=f(t)-\sum_{k=0}^{n-1}\frac{(t-a)^{\overline{k}}}{k!}\nabla^kf(a)
\end{equation}
and
\begin{equation}\label{ntrans2}
~_{b}\nabla^{-\alpha} ~_{b}^{C}\nabla^\alpha
f(t)=f(t)-\sum_{k=0}^{n-1}\frac{(b-t)^{\overline{k}}}{k!}~_{\ominus}\Delta^kf(b).
\end{equation}
In particular, if $0<\alpha\leq1$ then
\begin{equation}\label{ntrans3}
\nabla_a^{-\alpha} ~^{C}\nabla_a^\alpha f(t)= f(t)-f(a)~~\texttt{and}~~
~_{b}\nabla^{-\alpha} ~_{b}^{C}\nabla^\alpha f(t)=f(t)-f(b)
\end{equation}

\end{prop}

\begin{proof}
The proof of (\ref{ntrans1}) follows by the definition and applying Proposition \ref{lsemi} and (\ref{lbsde2}) of Proposition \ref{lbsdandds}. The proof of
(\ref{ntrans2}) follows by the definition and applying Proposition \ref{semi} and (\ref{bsde2}) of Proposition \ref{bsdandds}.
\end{proof}

Using the definition and Proposition \ref{power nabla left} and Proposition \ref{nabla right power}, we can find the nabla type Caputo fractional differences for certain power functions.
 For example, for $1 \neq \beta >0$ and $\alpha\geq 0$ we have
\begin{equation}\label{find left n}
~^{C}\nabla_a^\alpha (t-a)^{\overline{\beta-1}}=\frac{\Gamma(\beta)}   {\Gamma(\beta-\alpha)} (t-a)^{\overline{\beta-\alpha-1}}
\end{equation}
and
\begin{equation}\label{find right n}
~^{C}_{b}\nabla^\alpha (b-t)^{\overline{\beta-1}}=\frac{\Gamma(\beta)}   {\Gamma(\beta-\alpha)}(b-t)^{\overline{\beta-\alpha-1}}.
\end{equation}
However,
\begin{equation}\label{find3}
~^{C}\nabla_a^\alpha 1=~^{C}_{b}\nabla^\alpha1=0
\end{equation}
whereas
\begin{equation}\label{find4}
\nabla_a^\alpha 1=
\frac{(t-a)^{(-\alpha)}}{\Gamma(1-\alpha)},~~~_{b}\nabla^\alpha
1=\frac{(b-t)^{(-\alpha)}}{\Gamma(1-\alpha)}.
\end{equation}

In the above formulae (\ref{find left n}) and (\ref{find right n}), we apply the convention that dividing over a pole leads to zero. Therefore the fractional difference when $\beta-1=\alpha-j,~~~j=1,2,...,n$ is zero.

\begin{rem} \label{relate to} The results obtained in Theorem
\ref{relate}and afterward  agree with those in the usual
continuous case (See \cite{Kilbas} pages 91,96).
\end{rem}

\section{A dual nabla Caputo fractional difference}

In the previous section the nabla Caputo fractional difference is defined under the assumption that $f$ is known before $a$ in the left case and under the assumption that $f$ is known after $b$ in the right case. In this section we define other nabla Caputo fractional differences for which not necessary to request any information about $f$ before $a$ or after $b$. Since we shall show that these Caputo fractional differences are the dual ones for the delta Caputo fractional differences, we call them dual nabla Caputo fractional differences.

\begin{defn}
Let $\alpha >0,~ n=[\alpha]+1,~a(\alpha)=a+n-1$ and $b(\alpha)=b-n+1$. Then the dual nabla left and right Caputo fractional differences are defined by
\begin{equation}\label{Cdual left}
   ~^{C}\nabla_{a(\alpha)}^\alpha f(t)=\nabla_{a(\alpha)}^{-(n-\alpha)} \nabla^n f(t),~~t \in \mathbb{N}_{a+n}
\end{equation}

and
\begin{equation}\label{Cdual right}
  _{b(\alpha)}~ ^{C}\nabla^\alpha f(t)=~_{b(\alpha)}\nabla^{-(n-\alpha)} {\ominus}\Delta^n f(t), ~~t \in ~_{b-n}\mathbb{N},
\end{equation}
respectively.
\end{defn}
Notice that the Caputo and the dual Caputo differences coincide when $0<\alpha\leq 1$ and differ for higher order. That is for $0<\alpha\leq 1$
$$~^{C}\nabla_{a(\alpha)}^\alpha f(t)=~^{C}\nabla_a^\alpha f(t)~~and ~~ _{b(\alpha)} ^{C}\nabla^\alpha f(t)= ~_{b} ^{C}\nabla^\alpha f(t).$$

The following proposition states a dual relation between left delta Caputo fractional differences and left nabla (dual) Caputo fractional differences.

\begin{prop} \label{lCdual}
For $\alpha >0,~ n=[\alpha]+1,~a(\alpha)=a+n-1$, we have

\begin{equation}\label{lCdual one}
   (~^{C} \Delta_a^\alpha f)(t-\alpha)=  (~^{C} \nabla_{a(\alpha)}^\alpha f)(t),~~t \in N_{a+n}.
\end{equation}
\end{prop}
\begin{proof}
For $~~t \in N_{a+n}$, we have
\begin{eqnarray}\nonumber
    (~^{C} \Delta_a^\alpha f)(t-\alpha)&=& \frac{1}{\Gamma(n-\alpha)}\sum_{s=a} ^{t-n} (t-\alpha-\sigma(s))^{(n-\alpha-1)}\Delta^n f(s)\\ \nonumber
   &=& \frac{1}{\Gamma(n-\alpha)}\sum_{s=a} ^{t-n} (t-\alpha-\sigma(s))^{(n-\alpha-1)}\nabla^n f(s+n)\\
   &=&   \frac{1}{\Gamma(n-\alpha)}\sum_{r=a+n} ^{t} (t-\rho(r)+n-\alpha-2)^{(n-\alpha-1)}\nabla^n f(r)\\ \nonumber
  &=& \frac{1}{\Gamma(n-\alpha)}\sum_{r=a+n} ^{t} (t-\rho(r))^{\overline{n-\alpha-1}}\nabla^n f(r) \\ \nonumber
   &=&  (~^{C} \nabla_{a(\alpha)}^\alpha f)(t). \nonumber
\end{eqnarray}
\end{proof}
Analogously, the following proposition relates right delta Caputo fractional differences and  right nabla (dual) Caputo fractional differences.
\begin{prop} \label{rCdual}
For $\alpha >0,~ n=[\alpha]+1,~b(\alpha)=b-n+1$, we have

\begin{equation}\label{rCdual one}
   (~^{C} _{b}\Delta^\alpha f)(t+\alpha)=  (~^{C} _{b(\alpha)}\nabla^\alpha f)(t),~~t \in~ _{b-n}\mathbb{N}.
\end{equation}
\end{prop}

The following theorem modifies Theorem \label{LNg} when $f$ is only defined at $\mathbb{N}_a$.
 \begin{thm} \label{modleft}
 For any real number $\alpha$ and any positive integer $p$, the
following equality holds:
\begin{equation} \label{dLNg1}
\nabla_{a+p-1}^{-\alpha}~\nabla^p f(t)=\nabla^p
\nabla_{a+p-1}^{-\alpha}f(t)-\sum_{k=0}^{p-1}\frac{(t-(a+p-1))^{\overline{\alpha-p+k}}}
{\Gamma(\alpha+k-p+1)}\nabla^k f(a+p-1).
\end{equation}
where $f$ is defined on only $\mathbb{N}_a$ .
 \end{thm}
 The proof follows by applying   Remark \ref{lforany} inductively.

 \indent

 Similarly, in the right case we have
 \begin{thm}  \label{modright}
 For any real number $\alpha$ and any positive integer $p$, the
following equality holds:
 \begin{equation} \label{dRNg1}
~_{b-p+1}\nabla^{-\alpha}~_{\circleddash}\Delta^p f(t)=~_{\circleddash}\Delta^p
~_{b-p+1}\nabla^{-\alpha}f(t)-\sum_{k=0}^{p-1}\frac{(b-p+1-t)^{\overline{\alpha-p+k}}}
{\Gamma(\alpha+k-p+1)}~_{\ominus}\Delta^k f(b-p+1)
\end{equation}
where $f$ is defined on $_{b}N$ only.
 \end{thm}
\indent
 Now by using the modified Theorem \ref{modleft} and Theorem \ref{modright} we have

 \begin{thm} \label{nabla relate}
For any $\alpha>0$, we have
\begin{equation}\label{nrelate1}
~^{C}\nabla_{a(\alpha)}^\alpha f(t)=\nabla_{a(\alpha)}^\alpha f(t)-\sum_{k=0}^{n-1}
\frac{(t-a(\alpha))^{\overline{k-\alpha}}}{\Gamma(k-\alpha+1)}\nabla^k f({a(\alpha)})
\end{equation}
and
\begin{equation}\label{nrelate2}
~_{b(\alpha)}^{C}\nabla^\alpha f(t)=~_{b(\alpha)}\nabla^\alpha f(t)-\sum_{k=0}^{n-1}
\frac{(b(\alpha)-t)^{\overline{k-\alpha}}}{\Gamma(k-\alpha+1)}~_{\ominus}\Delta^k f(b(\alpha)).
\end{equation}
In particular, when $0<\alpha<1$, then $a(\alpha)=a$ and $b(\alpha) =b$ and hence we have
\begin{equation}\label{nrelate2}
~^{C}\nabla_a^\alpha f(t)=\nabla_a^\alpha f(t)-
\frac{(t-a)^{\overline{-\alpha}}}{\Gamma(1-\alpha)} f(a)
\end{equation}
and
\begin{equation}\label{nrelate4}
~_{b}^{C}\nabla^\alpha f(t)=~_{b}\nabla^\alpha f(t)-
\frac{(b-t)^{\overline{-\alpha}}}{\Gamma(1-\alpha)} f(b)
\end{equation}
\end{thm}

 \indent

 Also, by using the modified Theorem \ref{modleft} and Theorem \ref{modright} we have

 \begin{prop} \label{dnabla trans}
Assume $\alpha>0$ and $f$ is defined on suitable domains $\mathbb{N}_a$
and $_{b}\mathbb{N}$. Then
\begin{equation}\label{dntrans1}
\nabla_{a(\alpha)}^{-\alpha} ~^{C}\nabla_{a(\alpha)}^\alpha
f(t)=f(t)-\sum_{k=0}^{n-1}\frac{(t-a(\alpha))^{\overline{k}}}{k!}\nabla^kf(a(\alpha))
\end{equation}
and
\begin{equation}\label{dntrans2}
~_{b(\alpha)}\nabla^{-\alpha} ~_{b(\alpha)}^{C}\nabla^\alpha
f(t)=f(t)-\sum_{k=0}^{n-1}\frac{(b(\alpha)-t)^{\overline{k}}}{k!}~_{\ominus}\Delta^kf(b(\alpha)).
\end{equation}
In particular, if $0<\alpha\leq1$ then $a(\alpha)=a$   and $b(\alpha)=b$ and hence
\begin{equation}\label{dntrans3}
\nabla_a^{-\alpha} ~^{C}\nabla_a^\alpha f(t)= f(t)-f(a)~~\texttt{and}~~
~_{b}\nabla^{-\alpha} ~_{b}^{C}\nabla^\alpha f(t)=f(t)-f(b)
\end{equation}

\end{prop}

\section{Integration by parts for Caputo fractional differences}

In this section we state the integration by parts formulas for nabla fractional sums and differences obtained in \cite{THFer}, then use the dual identities to obtain delta integration by part formulas.

\begin{prop} \label{summation by parts}\cite{THFer}
For $\alpha>0$, $a,b \in \mathbb{R}$, $f$ defined on $\mathbb{N}_a$ and $g$ defined on $~_{b}\mathbb{N}$, we have

\begin{equation}\label{sum1}
    \sum_{s=a+1}^{b-1}g(s) \nabla_a^{-\alpha} f(s)=\sum_{s=a+1}^{b-1}f(s)~_{b}\nabla^{-\alpha}g(s).
\end{equation}
\end{prop}
\begin{proof}
By the definition of the nabla left fractional sum we have
\begin{equation}\label{sum2}
   \sum_{s=a+1}^{b-1}g(s) \nabla_a^{-\alpha} f(s)=\frac{1}{\Gamma(\alpha)}\sum_{s=a+1}^{b-1} g(s)\sum_{r=a+1}^{s}(s-\rho(r))^{\overline{\alpha-1}}f(r).
\end{equation}
If we interchange the order of summation we reach at ( \ref{sum1}).

\end{proof}

By the help of Theorem \ref{LNg}, Proposition \ref{lsemi},
(\ref{s1}) and that $\nabla_a^{-(n-\alpha)}f(a)=0$, the authors in \cite{THFer} obtained the following left important tools which lead to a nabla integration by parts formula for fractional differences.

\begin{prop} \label{lbsdandds}\cite{THFer}
For $\alpha >0$, and $f$ defined in a suitable domain $\mathbb{N}_a$, we have

\begin{equation} \label{lbdse}
\nabla_a^{\alpha}  \nabla_a^{-\alpha}f(t)=f(t),
\end{equation}

\begin{equation} \label{lbsde1}
\nabla_a^{-\alpha} \nabla_a^{\alpha} f(t)=f(t),~~\texttt{when}~
\alpha\notin\mathbb{N},
 \end{equation}
 and
 \begin{equation} \label{lbsde2}
 \nabla_a^{-\alpha}  \nabla_a^{\alpha}f(t)=
f(t)-\sum_{k=0}^{n-1}\frac{(t-a)^{\overline{k}}}{k!} \nabla^kf(a),
\texttt{,when}~\alpha=n \in \mathbb{N}.
\end{equation}
\end{prop}

 By the help of Theorem \ref{RNg}, Proposition \ref{semi},
(\ref{s2}) and that $~_{b}\nabla^{-(n-\alpha)}f(b)=0$, the authors also in \cite{THFer} also obtained the following right important tool:

\begin{prop} \label{bsdandds}\cite{THFer}
For $\alpha >0$, and $f$ defined in a suitable domain $~_{b}\mathbb{N}$, we have

\begin{equation} \label{bdse}~_{b}\nabla^{\alpha}  ~_{b}\nabla^{-\alpha}f(t)=f(t),
\end{equation}

\begin{equation} \label{bsde1}
~_{b}\nabla^{-\alpha} ~_{b}\nabla^{\alpha} f(t)=f(t),~~\texttt{when}~
\alpha\notin\mathbb{N},
 \end{equation}
 and
 \begin{equation} \label{bsde2}
~_{b} \nabla^{-\alpha}  ~_{b}\nabla^{\alpha}f(t)=
f(t)-\sum_{k=0}^{n-1}\frac{(b-t)^{\overline{k}}}{k!}~_{a}\Delta^kf(b)
\texttt{,when}~\alpha=n \in \mathbb{N}.
\end{equation}
\end{prop}

\begin{prop} \label{nabla bydiff} \cite{THFer}
Let $\alpha>0$ be non-integer and $a,b\in \mathbb{R}$ such that $a< b$ and $b\equiv
a~(mod~1)$.If $f$ is defined on $ _{b}N$ and $g$ is
defined on $N_a$, then
\begin{equation}\label{nabla bydiff1}
\sum_{s=a+1}^{b-1} f(s)\nabla_a^\alpha g(s)
=\sum_{s=a+1}^{b-1}g(s)~_{b}\nabla^\alpha f(s).
\end{equation}
\end{prop}

The proof was achieved by making use of Proposition \ref {summation by parts} and the tools Proposition \ref{lbsdandds} and Proposition \ref{bsdandds}.

By the above nabla integration by parts formulas and the dual identities in Lemma \ref{left dual} adjusted to our definitions and Lemma \ref{right dual}, in \cite{Thsh} the following delta integration by parts formulas were obtained.

\begin{prop} \label{delta by parts semmation}\cite{Thsh}
Let $\alpha>0$, $a,b\in \mathbb{R}$ such that $a< b$ and $b\equiv
a~(mod~1)$. If $f$ is defined on $N_a$ and $g$ is defined on
$_{b}N$, then we have
\begin{equation}\label{byse}
\sum_{s=a+1}^{b -1}g(s)(\Delta_{a+1}
^{-\alpha}f)(s+\alpha)=\sum_{s=a+1}^{b-1}f(s) ~_{b-1}\Delta ^{-\alpha}g(s-\alpha).
\end{equation}

\end{prop}

\begin{prop} \label{delta by parts semmation}\cite{Thsh}
Let $\alpha>0$ be non-integer and assume that $b\equiv a~(mod~1)$. If $f$ is defined on $ _{b}N$ and $g$ is
defined on $N_a$, then
\begin{equation}\label{bydiff1}
\sum_{s=a+1}^{b-1} f(s)\Delta_{a+1}^\alpha
g(s-\alpha)=\sum_{s=a+1}^{b-1}g(s)~_{b-1}\Delta^\alpha f(s+\alpha).
\end{equation}

\end{prop}

Now, we proceed in this section to obtain nabla and delta integration by parts formulas for Caputo fractional differences.

\begin{thm} \label{Caputo by parts}
Let $0<\alpha<1$ and $f,g$ be functions defined on $\mathbb{N}_a \cap ~_{b}\mathbb{N}$ where $a\equiv b ~(mod ~1)$. Then
\begin{equation} \label{cbp1}
\sum_{s=a+1}^{b+1} g(s) ~^{C}\nabla_a^\alpha f(s)=f(s) ~_{b}\nabla^{-(1-\alpha)}g(s)\mid_a^{b-1}+ \sum_{s=a}^{b-2} f(s) ~_{b}\nabla^\alpha g(s),
\end{equation}
where clearly $_{b}\nabla^{-(1-\alpha)}g(b-1)= g(b-1)$.
\end{thm}

\begin{proof}
From the definition of Caputo fractional difference and Proposition \ref{summation by parts} we have
\begin{equation}\label{sbp1}
 \sum_{s=a+1}^{b+1} g(s) ~^{C}\nabla_a^\alpha f(s)= \sum_{s=a+1}^{b+1} g(s) \nabla_a^{-(1-\alpha)} \nabla f(s)= \sum_{s=a+1}^{b+1} \nabla f(s)~_{b}\nabla^{-(1-\alpha)} g(s).
\end{equation}
By integration by parts from difference calculus, $\nabla f(s)=\Delta f(s-1)$ and the definition of nabla right fractional difference, we reach at
\begin{equation}
  \sum_{s=a+1}^{b+1} g(s) ~^{C}\nabla_a^\alpha f(s)= f(s) ~_{b}\nabla^{-(1-\alpha)} g(s)|_a^{b-1} +  \sum_{s=a+1}^{b+1} f(s-1)(~_{b}\nabla^\alpha g)(s-1).
\end{equation}
From which ( \ref{cbp1}) follows.
\end{proof}

\begin{thm} \label{Caputo by parts delta}
Let $0<\alpha<1$ and $f,g$ be functions defined on $\mathbb{N}_a \cap ~_{b}\mathbb{N}$ where $a\equiv b ~(mod ~1)$. Then
\begin{equation}\label{cbpd1}
\sum_{s=a+1}^{b+1} g(s) ~^{C}\Delta_a^\alpha f(s-\alpha)=f(s) ~_{b-1}\Delta^{-(1-\alpha)}g(s-(1-\alpha))\mid_a^{b-1}+ \sum_a^{b-2} f(s) ~_{b-1}\Delta^\alpha g(s+\alpha).
\end{equation}

\end{thm}
\begin{proof}
By the dual Caputo identity (\ref{lCdual one}) in Proposition \ref{lCdual} and Theorem \ref{Caputo by parts} we have

\begin{equation}\nonumber
  \sum_{s=a+1}^{b+1} g(s) ~^{C}\Delta_a^\alpha f(s-\alpha)=f(s) ~_{b}\nabla^{-(1-\alpha)}|_a^{b-1}+ \sum_{s=a}^{b-2} f(s)~ _{b}\nabla^\alpha g(s).
\end{equation}
Then (\ref{cbpd1}) follows by Lemma \ref{right dual} (i, ii).
\end{proof}

\section{The Q-operator and fractional difference equations}

If $f(s)$ is defined on $N_a\cap ~_{b}N$ and $a\equiv b~ (mod
~1)$ then $(Qf)(s)=f(a+b-s)$. The Q-operator generates a dual identity by which the left type and the right type fractional sums and differences are related. Using the change of variable $u=a+b-s$, in \cite{Th Caputo} it was shown  that
\begin{equation}\label{sum pr}
    \Delta_a^{-\alpha}Qf(t)= Q~_{b}\Delta^{-\alpha}f(t),
\end{equation}

and hence
\begin{equation}\label{pr}
    ~^{C}\Delta_a^\alpha Qf(t)= Q (~^{C}_{b}\Delta^\alpha f)(t).
\end{equation}
The proof of (\ref{pr}) follows by the definition, (\ref{sum pr}) and by noting that

$$-Q\nabla f(t)=\Delta Qf(t).$$

Similarly, in the nabla case we have
\begin{equation}\label{npr sum}
   \nabla_a^{-\alpha}Qf(t)= Q~_{b}\nabla^{-\alpha}f(t),
 \end{equation}

and hence

\begin{equation}\label{rpr}
   ~^{C}\nabla_a^\alpha Qf(t)= Q (~^{C}_{b}\nabla^\alpha f)(t).
\end{equation}

The proof of (\ref{rpr}) follows by the definition, (\ref{npr sum}) and that

$$-Q\Delta f(t)=\nabla Qf(t).$$

The Q-dual identities (\ref{pr}) and( \ref{rpr}) are  still valid for the delta and nabla  (Riemann) fractional differences, respectively. The proof is similar to the Caputo case above.

It is remarkable to mention that the Q-dual identity (\ref{pr})  can not be obtained if the definition of the delta right fractional difference introduced by Nuno R.O. Bastos et al. in \cite{Nuno} or by At{\i}c{\i} F. et al. in \cite{Atmodel} is used. Thus, the definition introduced in \cite{Th Caputo} and \cite{TDbyparts} is more convenience . Analogously, the Q-dual identity (\ref{rpr}) indicates that the nabla right Riemann and Caputo fractional differences presented in this article are also more convenient.

 It is clear from the above argument that, the Q-operator agrees with its continuous counterpart when applied to left and
right fractional Riemann Integrals and the Caputo and Riemann derivatives. More
generally, this discrete version of the Q-operator can be used to
transform the discrete delay-type fractional functional difference dynamic
equations to advanced ones. For details in the continuous counterparts see
\cite{Thabet}.

\begin{exam} \label{nlinear} \cite{Th Caputo}
Let $0<\alpha\leq 1$, $a=\alpha-1$ and consider the left Caputo
nonhomogeneous  fractional difference equation
\begin{equation} \label{lfractional}
~~^{C}\Delta_a^\alpha y(t)= \lambda y(t+\alpha-1)+f(t),~~y(a)=a_0,~t\in\mathbb{ N}_0.
\end{equation}
Note that the solution $y(t)$, if exists, is defined on $\mathbb{N}_a$ and hence
$~^{C}\Delta_a^\alpha y(t)$ becomes defined on $ \mathbb{N}_{a+(1-\alpha)}=\mathbb{N}_0$. Thus,
if we apply $\Delta_0^{-\alpha}$ on the equation (\ref{lfractional})  and then use the successive approximation the following delta type solution is obtained

\begin{equation} \label{fin3}
y(t)=a_0 E_{(\alpha)}(\lambda,t)+\sum_{s=0}^{t-\alpha}
E_{(\alpha,\alpha)}(\lambda,t-\sigma(s)) f(s)
\end{equation}

\end{exam}

where the delta type discrete Mittag-Leffler functions are defined by

\begin{defn} \label{DML} (Delta Discrete Mittag-Leffler) \cite{Th Caputo} For $\lambda \in \mathbb{R}$ and $\alpha, \beta, z \in \mathbb{C}$ with $Re(\alpha)>0$, the Delta discrete (like) Mittag-Leffler functions are defined by
\begin{equation}
E_{(\alpha, \beta)}(\lambda,z)= \sum_{k=0}^\infty \lambda^k
\frac{(z+(k-1)(\alpha-1))^{(k
\alpha)}(z+k(\alpha-1))^{( \beta-1)}}{\Gamma(\alpha
k+\beta)}.
\end{equation}

For $\beta=1$, it is written that
\begin{equation} \label{M22}
E_{(\alpha)} (\lambda, z)\triangleq E_{(\alpha, 1)}(\lambda, z)= \sum_{k=0}^\infty \lambda^k
\frac{(z+(k-1)(\alpha-1))^{(k
\alpha)}}{\Gamma(\alpha
k+1)}.
\end{equation}

\end{defn}

Next, we solve nabla Caputo type nonhomogeneous fractional difference equation to formulate nabla type discrete Mittag-Leffler functions
. Those functions generalize the nabla discrete exponential functions. For details about delta and nabla type exponential functions on arbitrary time scales we refer the reader to (\cite{Adv}, pages 10, 76).

\begin{defn} \label{nDML} (Nabla Discrete Mittag-Leffler)  For $\lambda \in \mathbb{R}$ and $\alpha, \beta, z \in \mathbb{C}$ with $Re(\alpha)>0$, the nabla discrete (like) Mittag-Leffler functions are defined by
\begin{equation}
E_{\overline{\alpha, \beta}}(\lambda,z)= \sum_{k=0}^\infty \lambda^k
\frac{z^{\overline{k\alpha+\beta-1}}} {\Gamma(\alpha
k+\beta)}.
\end{equation}
For $\beta=1$, it is written that
\begin{equation} \label{nM22}
E_{\overline{\alpha}} (\lambda, z)\triangleq E_{\overline{\alpha, 1}}(\lambda, z)=  \sum_{k=0}^\infty \lambda^k
\frac{z^{\overline{k\alpha}}} {\Gamma(\alpha
k+1)}.
\end{equation}

\end{defn}

\begin{exam} \label{nabla nlinear}
Let $0<\alpha\leq 1$, $a \in\mathbb{R}$ and consider the nabla left Caputo
nonhomogeneous  fractional difference equation
\begin{equation} \label{ nabla lfractional}
~^{C}\nabla_a^\alpha y(t)= \lambda y(t)+f(t),~~y(a)=a_0,~t\in \mathbb{N}_a.
\end{equation}

If we apply $\nabla_a^{-\alpha}$ on  equation (\ref{ nabla lfractional}) then
by (\ref{ntrans3}) we see that
$$y(t)= a_0+\frac{\lambda}{\Gamma(\alpha)}\sum_{s=a+1}^{t}(t-\rho(s))^{\overline{\alpha-1}}
y(s)+\nabla_a^{-\alpha}f(t).$$
To obtain an explicit  solution, we apply the method of successive
approximation. Set $y_0(t)=a_0$ and $$y_m(t)=a_0+ \nabla_a^{-\alpha}
[\lambda y_{m-1}(t)+f(t)], m=1,2,3,.... $$
For $m=1$, we have by the power formula (\ref{pnl1})
 $$y_1(t)=a_0[1+\frac{\lambda
(t-a)^{\overline{\alpha}}}{\Gamma(\alpha+1)}]+\nabla_a^{-\alpha}f(t).$$
For $m=2$, we also see by the help of Proposition \ref{lsemi} that

$$y_2(t)= a_0+ \lambda \nabla_a^{-\alpha}[a_0+
\frac{(t-a)^{\overline{\alpha}}}{\Gamma(\alpha+1)}]+\nabla_a^{-\alpha}f(t)+
\lambda \nabla_a^{-2 \alpha}f(t) =$$
$$a_0 [1+\frac{\lambda t^{(\alpha)}}{\Gamma(\alpha+1)}+\frac{\lambda^2
(t-a)^{\overline{2\alpha}}}{\Gamma(2\alpha+1)}]+\nabla_a^{-\alpha}f(t)+
\lambda \nabla_a^{-2 \alpha}f(t).$$
Proceeding inductively and making use of Proposition \ref{lsemi} and let
$m\rightarrow\infty$ we obtain the solution
$$y(t)=a_0\sum_{k=0}^\infty\frac{\lambda^k
(t-a)^{\overline{k\alpha}}}   {\Gamma(k\alpha+1)}]+\sum_{k=1}^\infty
\lambda ^{k-1} (\nabla_a^{-k \alpha}f)(t).$$
 Then,
\begin{equation} \label{nabla fin}
y(t)=a_0 E_{\overline{\alpha}}(\lambda,t-a)+\sum_{k=0}^\infty \lambda^k
\frac{1}{\Gamma(\alpha
k+\alpha)}\sum_{s=a+1}^{t}(t-\rho(s))^{\overline{k\alpha+\alpha-1}}f(s).
\end{equation}
Interchanging the order of sums in (\ref{nabla fin}), we conclude that

\begin{equation} \label{fin2}
y(t)=a_0 E_{\overline{\alpha}}(\lambda,t-a)+\sum_{s=a+1}^{t}\sum_{k=0}^\infty \lambda^k
\frac{(t-\rho(s))^{\overline{k\alpha+\alpha-1}}}{\Gamma(\alpha
k+\alpha)}f(s).
\end{equation}
That is
\begin{equation} \label{fin3}
y(t)=a_0 E_{\overline{\alpha}}(\lambda,t-a)+\sum_{s=a+1}^{t}
E_{\overline{\alpha,\alpha}}(\lambda,t-\rho(s)) f(s).
\end{equation}

\end{exam}

\begin{rem}
If we solve the nabla discrete fractional system (\ref{ nabla lfractional}) with $\alpha=1$ and $a_0=1$ we obtain the solution

$$y(t)= \sum_{k=0}^\infty \lambda^k \frac{(t-a)^{\overline{k}}} {k!}+\sum_{s=a+1}^t \sum_{k=0}^\infty \lambda^k \frac{(t-\rho(s))^{\overline{k}}} {k!} f(s).$$

The first part of the solution is the nabla discrete exponential function $\widehat{e}_{\lambda}(t,a)$. For the sake of more comparisons see (\cite{Adv}, chapter 3).
\end{rem}

\begin{exam} \label{Q} \cite{Th Caputo}
Let $0<\alpha\leq 1$, $a=\alpha-1$ and $b$ such that $a \equiv b ~(mod
~1)$. Let $y(t)$ be defined on $\mathbb{N}_a \cap~_{b}\mathbb{N}$. Consider the following
Caputo right fractional difference equation
\begin{equation} \label{rfractional}
~_{b}\Delta^\alpha Q y(t)=\lambda y(2a+b-t), + f(a+b-t)~~(Qy)(b)=a_0.
\end{equation}
If we apply the $Q$ operator on the  Caputo fractinal difference equation
(\ref{rfractional}), then we obtain the left Caputo fractional difference
equation (\ref{lfractional}). For more about the Q-operator and its use in functional fractional differential equations we refer to \cite{thabetetal}.
\end{exam}

Similar to Eaxample \ref{Q} above, we can use the Q-operator to transform, as well,  nabla left type fractional difference equations to right ones and vise versa.

\end{document}